\documentclass{article}
\usepackage{amsmath,amsthm,amsfonts,amssymb,mathtools}
\usepackage{fullpage}
\usepackage{graphicx}
\usepackage{pgf}
\usepackage{subcaption}
\usepackage{natbib}
\usepackage{xcolor}

\newtheorem{prop}{Proposition}
\newtheorem{thm}{Theorem}
\newtheorem{lemma}{Lemma}
\newtheorem{fact}{Fact}
\newtheorem{cor}{Corollary}

\theoremstyle{definition}\newtheorem{definition}{Definition}

\DeclareMathOperator*{\E}{\mathbb{E}}  

\newcommand{\nats}{\mathbb{N}}
\newcommand{\Ind}[1]{\mathbf{1}\left[#1\right]}
\newcommand{\Prob}[1]{\Pr{\left[#1\right]}}

\newcommand{\dist}{p}         
\newcommand{\distr}{(p)_r}         
\newcommand{\distq}{q}        
\newcommand{\balls}{m}        
\newcommand{\bins}{n}         
\newcommand{\binsr}{r}         
\newcommand{\load}{k}         
\newcommand{\wait}{T_{\load}} 
\newcommand{\maxload}{B_{\balls}} 
\newcommand{\maxloadr}{B_{\binsr,\balls}} 
\newcommand{\loadrp}{N_{\binsr+1,m}} 
\newcommand{\maxloadparam}[2]{\maxload^{#1,#2}}
\newcommand{\col}{C_{\load}} 
\newcommand{\bino}{\operatorname{Binom}} 

\newcommand{\loadnorm}{\|\dist\|_{\load}}
\newcommand{\qnorm}{\|\distq\|_{\load}}
\newcommand{\loadnormr}{\|\distr\|_{\load}}
\newcommand{\simplex}{\Delta_{\bins}}
\newcommand{\myrho}{\rho_{\balls,\load}}

\begin{document}
\title{Balls and Bins - Simple Concentration Bounds}
\author{Ernst Schulte-Geers \and Bo Waggoner}
\date{\today}
\maketitle

\begin{abstract}
  Concentration bounds are given for throwing balls into bins independently according to a distribution $\dist$.
  The probability of a $\load$-loaded bin after $\balls$ balls is shown to be controlled on both sides by $\myrho \coloneqq \frac{\balls \loadnorm}{\load}$.
  This gives concentration inequalities for the maximum load as well as for the waiting time until a $\load$-loaded bin.
\end{abstract}

\section{Introduction}
In a classic problem, balls are thrown randomly into $\bins$ bins.
Each ball selects a bin independently from a probability distribution $\dist$, where $\dist_i \geq 0$ and $\sum_{i=1}^{\bins} \dist_i = 1$.
The set of probability distributions on bins is denoted $\simplex$.
We seek non-asymptotic probability bounds on the max-loaded bin, in terms of two quantities:
\begin{itemize}
  \item $\maxload$, the maximum load after throwing $\balls$ balls (i.e. the maximum entry of a multinomial); and
  \item $\wait$, the number of trials until the maximum load is $\load$.
\end{itemize}

Naturally, much research in this vein exists (Section \ref{sec:related-work}), particularly for the uniform distribution.
This work contributes by giving nonasymptotic concentration bounds that aim to be user-friendly; and by showing that these bounds are controlled by the quantity
\begin{equation}
  \myrho \coloneqq \frac{\balls \loadnorm}{\load} . \label{eqn:rho}
\end{equation}
In particular, we observe a phase change in the likelihood of a $\load$-loaded bin as $\myrho$ changes by a small constant factor around $\myrho = 1$.
This gives concentration results for $\maxload$ around the quantity $\load^*$ where $\rho_{\balls,\load^*} = 1$; and for $\wait$ around $\balls^*$ where $\rho_{\balls^*,\load} = 1$.

\subsection{Results}
Let $\bino(\balls, \alpha)$ denote a Binomial random variable with $\balls$ trials and $\alpha$ probability of success.
We state most results in terms of $\maxload$, but a reformulation for $\wait$ is immediate by the observation that $\Pr[\wait \leq \balls] = \Pr[\maxload \geq \load]$.
\begin{thm} \label{thm:main}
  Let integers $\balls \geq 0$, $\bins \geq 1$, and $\load \geq 1$ be given.
  Let $\dist \in \simplex$ be given and throw balls into $\bins$ bins indepedently with distribution $\dist$.
  Let $\maxload$ be the maximum load after $\balls$ trials.
  Then
    \[ \Pr[\bino(\balls, \loadnorm) \geq \load] ~\leq~ \Pr[\maxload \geq \load] ~\leq~ {\balls \choose \load}\loadnorm^{\load} . \]
\end{thm}
The proofs of the upper and lower bounds appear in Sections \ref{sec:upper} and \ref{sec:lower} respectively.
Section \ref{sec:inf} generalizes the result to the case where $\bins = \aleph_{0}$ and we are interested in the maximum load of any subset of bins.
Illustrations of the bounds' tightness in simulations appear in Section \ref{sec:sims}.

We next leverage Theorem \ref{thm:main} to obtain more specific tail bounds.
The proofs of the following corollaries appear in Section \ref{sec:cor-proofs}.
Throughout, recall that $\Pr[\wait \leq \balls] = \Pr[\maxload \geq \load]$ and recall (\ref{eqn:rho}), the definition of $\myrho$.
\begin{cor} \label{cor:rho-tail}
  Throw $\balls \geq 1$ balls into $\bins \geq 2$ bins independently according to $\dist \in \simplex$.
  For any $\load \in [1,\infty)$:
  \begin{align*}
    \Prob{\maxload \geq \load} &\leq \exp\left(- \load ~ \ln(\tfrac{1}{e \cdot \myrho})\right)                  & \text{if $\myrho < 1/e$; and}  \\
    \Prob{\maxload \geq \load} &\geq 1 - \exp\left(- \load ~ \left(\myrho - \ln(e \cdot \myrho)\right) \right)  & \text{if $\myrho > 1$} .
  \end{align*}

\end{cor}

\begin{cor} \label{cor:rho-phase2}
  Throw $\balls \geq 1$ balls into $\bins \geq 2$ bins independently according to $\dist \in \simplex$.
  For any $\load \in [1,\infty)$ and $\delta \in (0,1)$:
  \begin{align*}
    \Prob{\maxload \geq \load} &\leq \delta      & \text{if $\myrho \leq \left(\frac{\delta^{1/\load}}{e}\right)$} \\
    \Prob{\maxload \geq \load} &\geq 1 - \delta  & \text{if $\myrho \geq \max\left\{e^2 , 2\ln(\tfrac{1}{\delta^{1/\load}}) \right\}$.} \\
  \end{align*}

\end{cor}

Finally, we derive concentration bounds for $\maxload$ and $\wait$.
\begin{cor} \label{cor:tail-bounds}
  Throw $\balls \geq 1$ balls into $\bins \geq 2$ bins independently according to $\dist \in \simplex$.
  Define $\load^* \in [1,\infty)$ to be the solution of $\frac{\load^*}{\|\dist\|_{\load^*}} = \balls$.
  For any $\delta \in (0,1)$:
  \begin{align}
    \Prob{ \maxload \geq \left(\frac{e}{\delta}\right) \load^* }
    &\leq \delta  & \text{and}  \\
    \Prob{ \maxload \geq \frac{1}{\max\left\{e^2 , 2\ln(\tfrac{1}{\delta}) \right\}} \load^* }
    &\geq 1 - \delta .
  \end{align}
\end{cor}

\begin{cor} \label{cor:tail-bounds-waiting}
  Throw balls into $\bins \geq 2$ bins independently according to $\dist \in \simplex$ and let $\load \in \nats, \load \geq 1$.
  Define $\balls^* = \frac{\load}{\loadnorm}$.
  Then for any $\delta \in (0,1)$:
  \begin{align}
    \Prob{\wait \leq \left(\frac{\delta}{e}\right) \balls^*}                     &\leq \delta  & \text{and} \\
    \Prob{\wait \leq \max\left\{e^2 , 2\ln(\tfrac{1}{\delta}) \right\} \balls^*} &\geq 1 - \delta  .
  \end{align}

\end{cor}

Finally, we observe that Theorem \ref{thm:main} directly implies simple bounds on $\E \wait$.
As discussed in Section \ref{sec:cor-proofs}, the upper bound is the best possible and the lower bound almost so (for completeness, we will mention a slightly better one).
\begin{cor} \label{cor:expected-wait}
  Let $\bins \geq 1$, $\dist \in \simplex$, and $\load \in \nats, \load \geq 1$.
  Then
    \[ \left(\frac{1}{e}\right) \left(\frac{\load}{\load + 1}\right) \frac{\load}{\loadnorm}
       \leq \E \wait
       \leq \frac{\load}{\loadnorm} . \]
\end{cor}

\paragraph{Remarks.}
Depending on the application, the concentration bound for $\maxload$ (Corollary \ref{cor:tail-bounds}) may or may not be satisfactory.
For small, e.g. single-digit maximum loads $\load^*$, a constant-factor interval may be considered too wide.
And even for large $\load^*$, true confidence interval around $\maxload$ could be as small as width zero in some scenarios, making a constant-factor range comparatively loose.\footnote{%
  We used the naive bound $\|\dist\|_{\load} \geq \|\dist\|_{\load+1}$.
  If this is far from tight, then $\myrho \gg \rho_{\balls,\load+1}$, implying that a confidence interval need not contain both $\load$ and $\load+1$.
  On the other hand, if it is close to tight, then $\dist$ is in a sense far from uniform, and the maximum bin may exhibit some concentration for this reason.
  Perhaps future work can leverage this tradeoff to give tighter bounds.}

We expect the confidence interval for $\wait$, Corollary \ref{cor:tail-bounds-waiting}, to be sufficient and useful for many scenarios.
It gives a closed-form $\balls^*$ for approximately how many balls must be thrown until some bin receives $\load$ balls; and shows that within a small constant factor around $\balls^*$, the probability transitions from almost $0$ to almost $1$.

More generally, Corollary \ref{cor:rho-phase2} suggests a kind of phase transition for $\Prob{\maxload \geq \load}$ around $\myrho = 1$.
In Section \ref{sec:sims}, we use simulation results to illustrate this transition along with the tightness of Theorem \ref{thm:main}.

As $\myrho \to \infty$, we observe $\Pr[\maxload < \load] \to 0$ exponentially fast in $\myrho$.
However, as $\myrho \to 0$, we only obtain $\Pr[\maxload \geq \load] \to 0$ polynomially in $\tfrac{1}{\myrho}$.
Given that the upper bound, $\Pr[\maxload \geq \load] \leq {\balls \choose \load} \loadnorm$, arises from a simple Markov inequality (we will see), one might expect it to be loose.
However, in this regime the upper and lower bounds match up to a factor only depending on $\load$.
Specifically, consider the regime $\myrho \leq \tfrac{1}{2}$, i.e. $\balls \loadnorm \leq \tfrac{\load}{2}$.
Suppose for example that $\loadnorm \leq \tfrac{1}{2}$, then we obtain
\begin{align*}
  {\balls \choose \load} \loadnorm^{\load}
  &\geq \Pr[\maxload \geq \load]  \\
  &\geq \Pr[\bino(\balls,\loadnorm)=\load]  \\
  &\geq (1 - \loadnorm)^{\balls} {\balls \choose \load} \loadnorm^{\load}  \\
  &\geq \exp\left(-2 \balls \loadnorm\right) {\balls \choose \load} \loadnorm^{\load}  \\
  &\geq e^{-\load}  {\balls \choose \load} \loadnorm^{\load} .
\end{align*}
If $\balls \loadnorm \to 0$, then the upper and lower bounds match to a factor $1 \pm o(1)$; this is illustrated for $\load=2$ in Figure \ref{fig:birthday}.

\paragraph{Applications with the uniform distribution.}
For any $\dist$ and $\balls$, a rough estimate of the max loaded bin can be obtained simply by solving $\myrho \approx 1$ for $\load$ (though this calculation is not always simple).
By Corollary \ref{cor:tail-bounds}, with high probability this estimate will be within a constant factor of $\maxload$.
Similarly, for any $\dist$ and $\load$, an estimate of the number of throws until a $\load$-loaded bin can be obtained by the calculation $\myrho \approx 1$ (which is simple).
In particular when $\dist$ is uniform, then $\loadnorm = \frac{\bins^{1/\load}}{\bins}$ and
  \[ \myrho = \frac{\balls \bins^{1/\load}}{\bins \load} .  \]

We recover some known results, possibly with more application-friendly tail bounds than have been previously stated.
\begin{itemize}
  \item For the classic birthday paradox, set $\load = 2$, obtaining $\myrho = \frac{\balls}{2\sqrt{\bins}}$ and $\balls^* = 2 \sqrt{\bins}$.
        Apply Corollary \ref{cor:tail-bounds-waiting} to obtain that $\tfrac{2 \delta}{e} \sqrt{\bins} \leq \wait \leq \max\{2e^2, 4\ln(\tfrac{1}{\delta})\} \sqrt{\bins}$ except with probability $2\delta$.
        The polynomial dependence on $\delta$ in the lower bound appears necessary (cf. above remarks and Figure \ref{fig:birthday}).
  \item In the classic case of $\balls = \bins$, the problem $\myrho = 1$ reduces to solving $\load^{\load} = \bins$, or $\load \approx \frac{\ln(\bins)}{\ln( \ln(\bins))}$.
  \item For another example, if $\balls = \bins \log(\bins)$, then for $\load = \log(\bins)$ we have $\myrho = 1$ exactly.
  \item Similarly, consider any $\balls \geq \bins \log(\bins)$.
        Then $\load^* \leq e \tfrac{\balls}{\bins}$, by observing that for $\load = e \cdot \tfrac{\balls}{\bins}$ we have $\myrho \leq 1$.
        As $\load^* \leq e \tfrac{\balls}{\bins}$, the maximum load is concentrated within a small constant factor of the average load.
\end{itemize}

\subsection{Related work} \label{sec:related-work}
The waiting time $T_\load$ for the first $\load$-fold repeat is subject of the classical $\load$-birthday problem and has been intensively investigated. The literature is vast, we just give some pointers. 
Most of the research focuses on asymptotic results for the case of uniformly distributed birthdays (bins).  \cite{von1939aufteilungs} started by proving Poisson limits for the number of $\load$-bins, and \cite{klamkin1967extensions} gave an elegant integral representation for $\mathbb{E}(T_\load)$ and proved asymptotic Weibull distribution
of $T_\load$. \citet{raab1998balls} gives precise estimates for the max-loaded bin in a variety of regimes.
However, these results focus on the uniform distribution only and are asymptotic. For the non-uniform case the possible limit distributions for $T_\load$ were given by \cite{camarri2000limit} ($\load=2$) resp. \cite{camarri1998asymptotics} ($\load>2$).

For obtaining bounds multinomial probabilities must be estimated, and several common approaches are known.
For the upper bound, a common basic approach is to separately bound the chance that each bin is $\load$-loaded and apply a union bound (Bonferroni correction).
For the lower bound, one can utilize \emph{negative dependence}~\citep{dubhashi1996balls}, based on the inequality of Mallows.
Another general technique is to approximate the process by independent Poisson variables in each bin.
This work differs from such results in that it emphasizes the role of $\myrho$ in controlling the probabilities, and that it gives nonasymptotic bounds.

For non-uniform $\dist$, non-asymptotic bounds for the case $\bins=2$ were studied by \cite{wiener2005bounds}, and the lower bound of  Theorem \ref{thm:main} for the case $\load=2$ was already proved there.
\cite{holst1995general} studies the general  $k$-birthday 
problem and shows that $\Pr[\wait \leq \balls]$ is Schur-convex in $\dist$. Thus $\Pr[\wait \leq \balls]$ is for a given 
$\loadnorm=c$ minimized resp. maximized for the distributions which are minimal resp. maximal (in the order of majorization) under this side condition.
A possible route to bounds then consists in determining these minimal/maximal distributions and to proceed from there. This is in essence Wiener's approach. The same idea could be followed for $\bins>2$, but the arguments become lengthy and tedious. In contrast we derive bounds here by elementary direct arguments.

\section{Upper bound} \label{sec:upper}
Recall that $\maxload$ is the maximum load when $\balls \geq 0$ balls are thrown independently into $\bins \geq 1$ bins with distribution $\dist \in \simplex$.
Let $\load \geq 1$.
\begin{prop} \label{prop:collision-upper}
  $\Pr[\maxload \geq \load] \leq {\balls \choose \load} \loadnorm^{\load}.$
\end{prop}
\begin{proof}
  Let $Z_{\load} \subseteq \{1,\dots,\balls\}$ be the set of all subsets of $\load$ distinct balls.
  For $S \in Z_{\load}$, let $\Ind{S}$ be the indicator that all balls in $S$ fall into the same bin; we call the event $\{\Ind{S}=1\}$ a \emph{$\load$-way collision}.
  Define $\col = \sum_{S \in Z_{\load}} \Ind{S}$, the number of $\load$-way collisions.

  For each $S \in Z_{\load}$, we have $\E[\Ind{S}] = \sum_{i=1}^{\bins} \dist_i^{\load} = \loadnorm^{\load}$.
  This follows because the event $\{\Ind{S}=1\}$ is the disjoint union of the events that all balls fall into bin $1$ or \dots or $\bins$.

  Thus, $\E[\col] = |Z_{\load}| \loadnorm^{\load} = {\balls \choose \load} \loadnorm^{\load}$.
  Now,
  \begin{align*}
    \Pr[\maxload \geq \load]
    &=    \Pr[\col \geq 1]  \\
    &\leq \E\left[ \col \right]  & \text{Markov's inequality}  \\
    &=    {\balls \choose \load} \loadnorm^{\load} .
  \end{align*}
\end{proof}

\section{Lower bound} \label{sec:lower}
Recall that $\maxload$ is the maximum load when $\balls \geq 0$ balls are thrown independently into $\bins \geq 1$ bins with distribution $\dist \in \simplex$.
Let $\load \geq 1$.
Write $[\bins] = \{1,\dots,\bins\}$.
In this section, we prove:
\begin{prop} \label{prop:lower-bound}
  $\Pr[\maxload \geq \load] \geq \Pr[\bino(\balls, \loadnorm) \geq \load].$
\end{prop}
The proof will modify $\dist$ to some $\dist'$ and consider the maximum load of only bins $2,\dots,\bins$.
By repeating the argument, we eventually compare to the maximum load under some $\distq$ of only the final bin, and we show that $\distq_{\bins} = \loadnorm$.

To execute this argument, we utilize a stand-alone proof of Proposition \ref{prop:lower-bound} for the $\bins=2$ case, starting with a useful known fact (e.g.~\cite{knuth2019art}, Vol. 4, Fascicle 5, Mathematical Preliminaries Redux, Problem 23):
\begin{fact} \label{fact:binom-rewrite}
  Let $0 \leq \alpha \leq 1$ and $0 \leq t \leq \balls-1$.
  Then $\Pr[\bino(\balls,\alpha) \leq t] = (1-\alpha)^{\balls-t} \sum_{j=0}^t {\balls-t-1 + j \choose j} \alpha^j$.
\end{fact}
\begin{proof}
  Suppose we flip coins independently with success probability $\alpha$ until we have $\balls - t$ tails, then halt.
  The probability that we halt at or before step $\balls$ is $\Pr[\bino(\balls,\alpha) \leq t]$.
  Now with a counting argument: Consider the steps where it is possible to halt, $\ell = \balls - t, \dots, \balls$.
  We stop at such an $\ell$ if and only if there are $\balls - t - 1$ tails in the first $\ell-1$ steps and a tail on the $\ell$th.
  The number of ways this can occur is ${\ell-1 \choose \balls-t-1} = {\ell-1 \choose \ell-(\balls-t)}$.
  The probability of each of these ways is $(1-\alpha)^{\balls-t} \alpha^{\ell - (\balls-t)}$.
  Thus the total probability of stopping is $\sum_{\ell=\balls-t}^{\balls} {\ell - 1 \choose \ell - (\balls - t)} (1 - \alpha)^{\balls - t} \alpha^{\ell - (\balls - t)}$.
  A change of variables with $j = \ell - (\balls -t)$ proves the fact.
\end{proof}

\begin{lemma} \label{lemma:lower-case-2}
  Given $\balls \geq 0$, $\bins=2$, $\dist \in \simplex$, and $\load \geq 1$, let $\dist' = \left(1-\loadnorm, \loadnorm\right)$ and let $\maxload'$ be the load of bin $2$ when throwing $\balls$ balls independently according to $\dist'$.
  Then $\Pr[\maxload \geq \load] \geq \Pr[\maxload' \geq \load]$.
\end{lemma}
\begin{proof}
  Let us first consider corner cases.
  If $\balls \leq \load-1$ then both sides are zero and the inequality is satisfied.
  Otherwise, if $\dist_1 \in \{0,1\}$ or if $\balls \geq 2\load - 1$ then the left side equals one and the inequality is satisfied.

  Now assume $0 < \dist_1 < 1$ and $\load \leq \balls \leq 2\load-2$.
  This implies $0 \leq \balls-\load \leq \load-2$.
  Let $X \sim \bino(\balls,\dist_1)$ and $Y \sim \bino(\balls,\dist_2)$.
  \begin{align}
    \Pr[\maxload \geq \load]
    &= \Pr[X \geq \load] + \Pr[X \leq \balls - \load]  & \text{disjoint events}  \nonumber \\
    &= \Pr[Y \leq \balls - \load] + \Pr[X \leq \balls - \load]  \nonumber \\
    &= \sum_{j=0}^{\balls-\load} {\load - 1 + j \choose j} \left( \dist_2^{\load} \dist_1^j + \dist_1^{\load} \dist_2^j \right)
       &\text{using Fact \ref{fact:binom-rewrite}.} \label{eqn:case2-lhs}
  \end{align}
  Meanwhile, let $X' \sim \bino(\balls, 1 - \loadnorm)$:
  \begin{align}
    \Pr[\maxload' \geq \load]
    &= \Pr[X' \leq \balls - \load]  \nonumber \\
    &= \sum_{j=0}^{\balls - \load} {\load - 1 + j \choose j} \loadnorm^{\load} (1 - \loadnorm)^{j} . \label{eqn:case2-rhs}
  \end{align}
  By inequality of $\load$-norm and $\infty$-norm, we have $\max\{\dist_1,\dist_2\} \leq \loadnorm$, so $1 - \loadnorm \leq \min\{\dist_1,\dist_2\}$.
  So
  \begin{align*}
    \loadnorm^{\load} (1 - \loadnorm)^j
    &=    \left(\dist_1^{\load} + \dist_2^{\load}\right) (1-\loadnorm)^j  \\
    &\leq \dist_1^{\load} \dist_2^j + \dist_2^{\load} \dist_1^j .
  \end{align*}
  Thus (\ref{eqn:case2-lhs}) is at least (\ref{eqn:case2-rhs}), as claimed.
\end{proof}

We will use Lemma \ref{lemma:lower-case-2} to iteratively compare the following variables.
\begin{definition}[$\maxloadparam{\distq}{j}$]
  When throwing $\balls$ balls into $\bins$ bins independently according to $\distq \in \simplex$, for any $j \in [\bins]$, define $\maxloadparam{\distq}{j}$ to be the maximum load of any bin in $\{j,\dots,\bins\}$.
  In particular, $\maxload = \maxloadparam{\dist}{1}$.
\end{definition}

\begin{lemma} \label{lemma:modify-step}
  Given $\balls, \bins \geq 2, \load$, any $\distq \in \simplex$, and any $j \in \{1,\dots,\bins-1\}$, define $\distq'$ as follows:
  $\distq'_{\ell} = \distq_{\ell}$ for all $\ell \not\in \{j,j+1\}$, and
  \begin{align*}
    \distq'_j     &= \distq_j + \distq_{j+1} - \|(\distq_j, \distq_{j+1})\|_{\load}  \\
    \distq'_{j+1} &= \|(\distq_j, \distq_{j+1})\|_{\load} .
  \end{align*}
  Then $\Pr[\maxloadparam{\distq}{j} < \load] \leq \Pr[\maxloadparam{\distq'}{j+1} < \load]$.
\end{lemma}
\begin{proof}
  We will condition on $t$ balls total landing in bins $\{j,j+1\}$.
  To that end, let $A_t$ be the maximum load when throwing $t$ balls into bins $\{j,j+1\}$ according to the conditional distribution $(\distq_j,\distq_{j+1})/(\distq_j+\distq_{j+1})$.
  Let $A'_t$ be the load of bin $j+1$ when throwing $t$ balls into $\{j,j+1\}$ according to $(\distq'_j,\distq'_{j+1})/(\distq'_j + \distq'_{j+1})$.
  Lemma \ref{lemma:lower-case-2} implies $\Pr[A_t < \load] \leq \Pr[A'_t < \load]$.

  Similarly, let $D_{\balls-t}$ (respectively, $D'_{\balls-t}$) be the maximum load of bins $\{j+2,\dots,\bins\}$ when throwing $\balls-t$ balls according to $\distq$ (respectively, $\distq'$) conditioned on landing in $[\bins] \setminus \{j,j+1\}$.
  (If $j = \bins-1$, define $D_{\balls-t} = D'_{\balls-t} = 0$.)
  We observe that $D_{\balls-t}$ and $D'_{\balls-t}$ are identically distributed because $\distq$ and $\distq'$ are identical outside of $\{j,j+1\}$.
  Let $f(t) = \Pr[ \bino(\balls,\distq_i+\distq_j) = t]$.
  Then
  \begin{align*}
    \Pr[\maxloadparam{\distq}{j} < \load]
    &=    \sum_{t=0}^{\balls} f(t) \Pr[A_t < \load]  \Pr[D_{\balls-t} < \load]  \\
    &=    \sum_{t=0}^{\balls} f(t) \Pr[A_t < \load]  \Pr[D'_{\balls-t} < \load]  & \text{observed above}  \\
    &\leq \sum_{t=0}^{\balls} f(t) \Pr[A'_t < \load] \Pr[D'_{\balls-t} < \load]  & \text{Lemma \ref{lemma:lower-case-2}}  \\
    &=    \Pr[\maxloadparam{\distq'}{j+1} < \load] .
  \end{align*}
\end{proof}

\begin{proof}[Proof of Proposition \ref{prop:lower-bound}]
  We show $\Pr[\maxload < \load] \leq \Pr[\bino(\balls, \loadnorm) < \load]$.
  If $\bins=1$ the result is immediate.

  Apply Lemma \ref{lemma:modify-step} to $\dist$ at $j=1$.
  We obtain $\dist'$ where $\dist'_2 = \|(\dist_1,\dist_2)\|_{\load} = \left(\dist_1^{\load} + \dist_2^{\load}\right)^{1/\load}$ and the guarantee $\Pr[\maxload < \load] \leq \Pr[\maxloadparam{\dist'}{2} < \load]$.
  Next apply Lemma \ref{lemma:modify-step} to $\dist'$ at $j=2$.
  We obtain $\dist''$ where $\dist''_3 = \|(\dist'_2, \dist_3)\|_{\load} = \left(\dist_1^{\load} + \dist_2^{\load} + \dist_3^{\load}\right)^{1/\load} = \|(\dist_1,\dist_2,\dist_3)\|_{\load}$ and the guarantee $\Pr[\maxload < \load] \leq \Pr[\maxloadparam{\dist''}{3} < \load]$.

  Continuing, by induction, at the final step we obtain some $\distq$ with $\distq_{\bins} = \loadnorm$ and the guarantee $\Pr[\maxload < \load] \leq \Pr[\maxloadparam{\distq}{\bins} < \load]$.
  Now we observe that $\maxloadparam{\distq}{\bins}$ is simply the load of the $\bins$th bin, so it is distributed $\bino(\balls, \loadnorm)$.
\end{proof}

\paragraph{Remark.}
The above elementary proof is based on a more concise version that utilizes generating functions to analyze $\wait$.
This proof appeared at \citep{esg2021bounding}.

\section{Illustrations} \label{sec:sims}
We illustrate Theorem \ref{thm:main} with some simulation results for the uniform distribution $\dist$, Figure \ref{fig:sim-unif}, and a non-uniform distribution $\dist$ proportional to $(1,2,\dots,\bins)$, Figure \ref{fig:sim-nonunif}.
Each curve plots the empirical probability of a $\load$-loaded bin as we vary the number of balls $\balls$ thrown.
The horizontal axes are parameterized by $\myrho = \frac{\balls \loadnorm}{\load}$.

We have derived upper and lower bounds with no dependence on the number of bins $\bins$ or the details of $\dist$ except through $\loadnorm$.
Simulations support that this is almost exactly true of the exact frequency as well:
\emph{(a)} In both Figures \ref{fig:sim-unif} and \ref{fig:sim-nonunif}, the top two plots are almost identical to the bottom two plots, suggesting that $\bins$ has almost no impact on $\Pr[\maxload \geq \load]$ once we control for $\myrho$.
\emph{(b)} All plots in Figure \ref{fig:sim-nonunif} are almost identical to their counterparts in Figure \ref{fig:sim-unif}, suggesting that the structure of $\dist$ has little impact on $\Pr[\maxload \geq \load]$ once we control for $\myrho$.

\begin{figure}[h!]
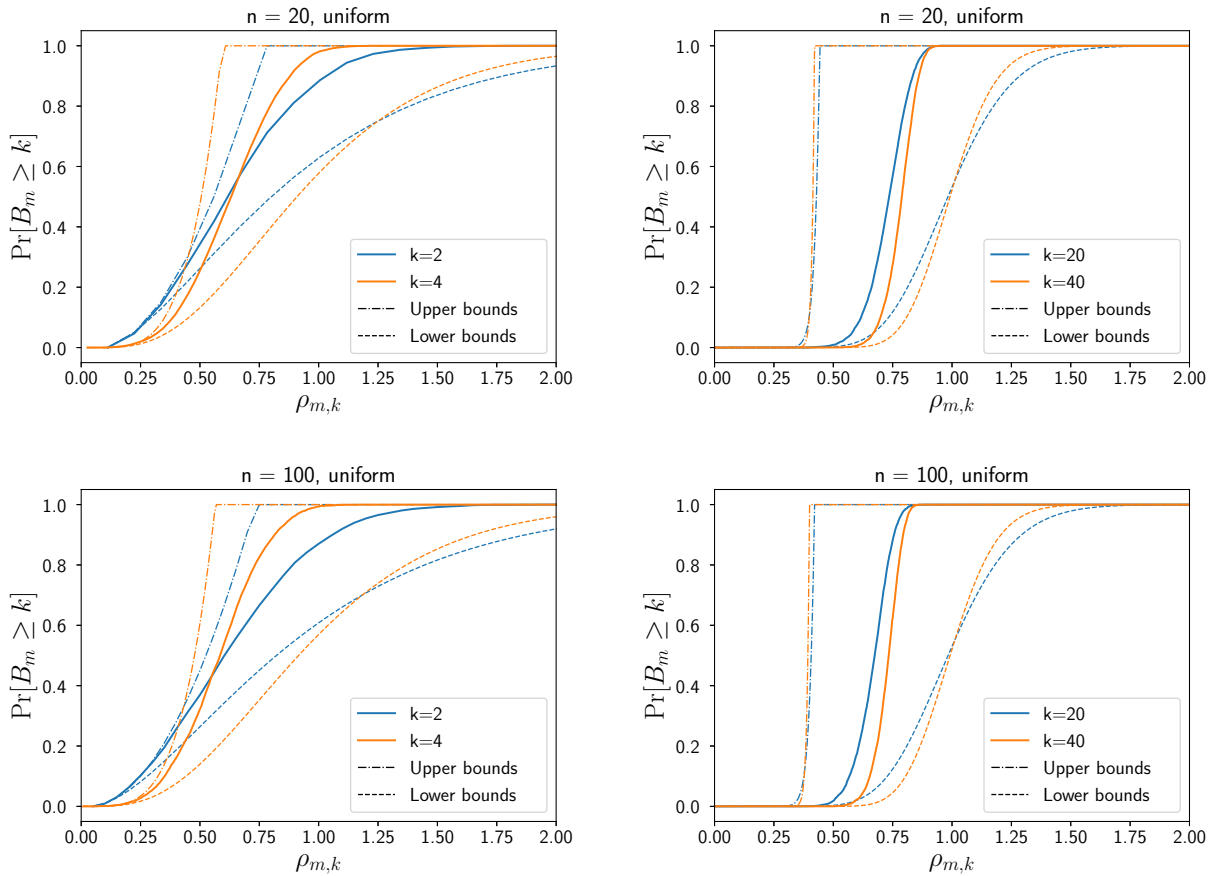

  \caption{\textbf{Uniform distribution:} Throwing $\balls$ balls into $\bins$ bins uniformly. For each fixed $\load$, we vary $\balls$ and plot the empirical frequency for $\Pr[\maxload \geq \load]$ across 5000 trials (solid lines) versus $\myrho = \frac{\balls \|\dist\|_{\load}}{\load}$. Dashed lines plot the upper and lower bounds of Theorem \ref{thm:main}. In all settings, but particularly as $\load$ grows, we observe a phase transition where the chance of a $k$-loaded bin quickly transitions from almost zero to almost one.}
  \label{fig:sim-unif}
  \centering
  \begin{subfigure}[b]{0.49\textwidth}
    \centering
    \resizebox{\textwidth}{!}{\input{figs/n20-k2,4-unif.pgf}}
  \end{subfigure}
  \hfill
  \begin{subfigure}[b]{0.49\textwidth}
    \centering
    \resizebox{\textwidth}{!}{\input{figs/n20-k20,40-unif.pgf}}
  \end{subfigure}

  \begin{subfigure}[b]{0.49\textwidth}
    \centering
    \resizebox{\textwidth}{!}{\input{figs/n100-k2,4-unif.pgf}}
  \end{subfigure}
  \hfill
  \begin{subfigure}[b]{0.49\textwidth}
    \centering
    \resizebox{\textwidth}{!}{\input{figs/n100-k20,40-unif.pgf}}
  \end{subfigure}
\end{figure}

\begin{figure}[h!]
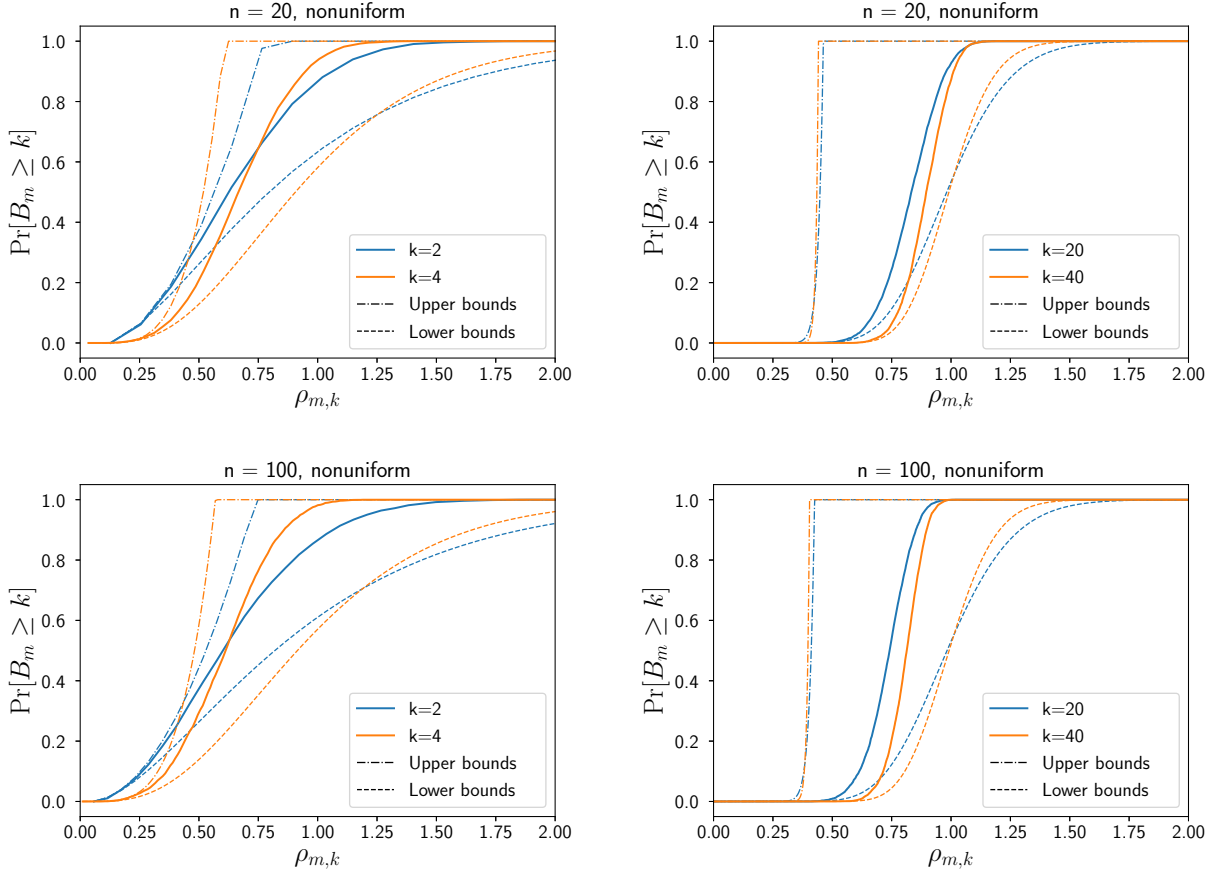

  \caption{\textbf{Nonuniform:} Throwing $\balls$ balls into $\bins$ bins according to the distribution $\tfrac{2}{\bins(\bins+1)}\left(1,\dots,\bins\right)$. Solid lines are empirical frequencies for $\Pr[\maxload \geq \load]$ from 5000 trials. Dashed lines plot the upper and lower bounds of Theorem \ref{thm:main}.}
  \label{fig:sim-nonunif}
  \centering
  \begin{subfigure}[b]{0.49\textwidth}
    \centering
    \resizebox{\textwidth}{!}{\input{figs/n20-k2,4-nonunif.pgf}}
  \end{subfigure}
  \hfill
  \begin{subfigure}[b]{0.49\textwidth}
    \centering
    \resizebox{\textwidth}{!}{\input{figs/n20-k20,40-nonunif.pgf}}
  \end{subfigure}

  \begin{subfigure}[b]{0.49\textwidth}
    \centering
    \resizebox{\textwidth}{!}{\input{figs/n100-k2,4-nonunif.pgf}}
  \end{subfigure}
  \hfill
  \begin{subfigure}[b]{0.49\textwidth}
    \centering
    \resizebox{\textwidth}{!}{\input{figs/n100-k20,40-nonunif.pgf}}
  \end{subfigure}
\end{figure}

\begin{figure}
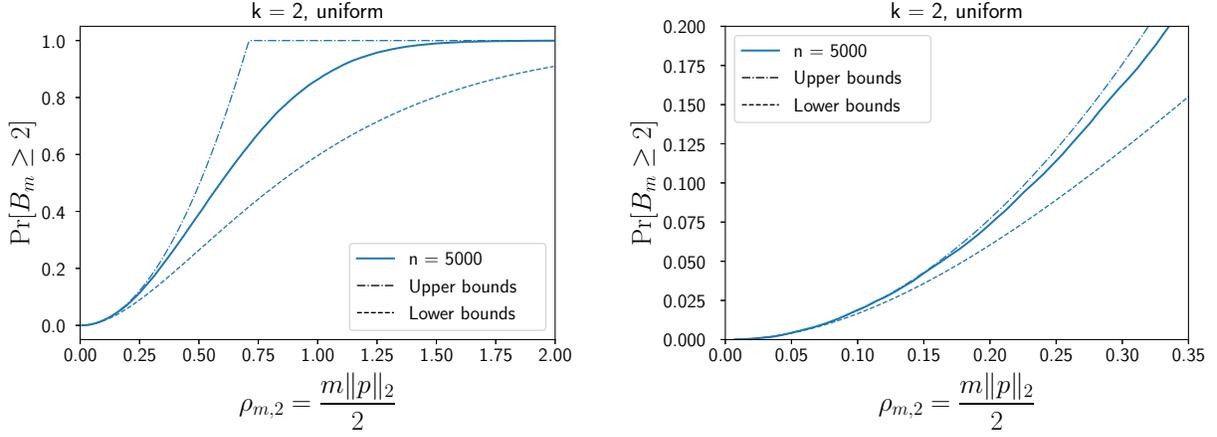

  \caption{\textbf{Birthday paradox:} Illustration of the the probability of a $2$-loaded bin when throwing $\balls$ balls uniformly into $\bins=5000$ bins. Solid lines are empirical frequencies for $\Pr[\maxload \geq 2]$ from 50000 trials. Dashed lines plot the upper and lower bounds of Theorem \ref{thm:main}. The bounds match to a ratio of $1 \pm o(1)$ as $\rho_2 \to 0$.}
  \label{fig:birthday}
  \centering
  \begin{subfigure}[b]{0.49\textwidth}
    \centering
    \resizebox{\textwidth}{!}{\input{figs/n5000-k2-unif.pgf}}
  \end{subfigure}
  \hfill
  \begin{subfigure}[b]{0.49\textwidth}
    \centering
    \resizebox{\textwidth}{!}{\input{figs/n5000-k2-unif-zoomed.pgf}}
  \end{subfigure}

\end{figure}

\section{Proofs of Corollaries} \label{sec:cor-proofs}

It will be useful to observe that $\myrho := \frac{\balls \loadnorm}{\load}$ is decreasing in $\load$.

\begin{proof}[Proof of Corollary \ref{cor:rho-tail}]

  The upper bound follows from the inequality\footnote{Proof of inequality: ${\balls \choose \load} \leq \tfrac{\balls^{\load}}{\load!}$, and from $e^{\load} = \sum_{i=0}^{\infty} \tfrac{\load^i}{i!} \geq \tfrac{\load^{\load}}{\load!}$ we get $\frac{1}{\load!} \leq \tfrac{e^{\load}}{\load^{\load}}$.} ${\balls \choose \load} \leq \left(\frac{e \cdot \balls}{\load}\right)^{\load}$, giving for integer $\load \geq 1$, $\Pr[\maxload \geq \load] \leq \delta_{\load} := \exp\left(- \load \ln \frac{e}{\myrho}\right)$.
  Because $\delta_{\load}$ is decreasing in $\load$, the bound follows for all real $\load \geq 1$: we have $\Pr[\maxload \geq \load] = \Pr[\maxload \geq \lceil \load \rceil] \leq \delta_{\lceil \load \rceil} \leq \delta_{\load}$.

  For the lower bound: by \citet{hoeffding1962probability},
  \[ \Pr[\bino(\balls, \loadnorm) < \load] \leq \exp\left(- \balls D(\tfrac{\load}{\balls}, \loadnorm) \right) \]
  where
  \begin{align*}
    \balls D(\tfrac{\load}{\balls}, \loadnorm)
    &= \load \ln\left(\frac{\load}{\balls \loadnorm}\right) + (\balls - \load)\ln\left(\frac{1-\tfrac{\load}{\balls}}{1-\loadnorm}\right)  \\
    &= \load \ln\left(\frac{1}{\myrho}\right) - (\balls - \load)\ln\left(\frac{\balls - \balls \loadnorm}{\balls - \load}\right)  \\
    &= \load \ln\left(\frac{1}{\myrho}\right) - (\balls - \load)\ln\left(1 + \frac{\load - \balls \loadnorm}{\balls - \load}\right)  \\
    &\geq \load \ln\left(\frac{1}{\myrho}\right) - (\balls - \load)\left(\frac{\load - \balls \loadnorm}{\balls - \load}\right)  \\
    &= \load \ln\left(\frac{1}{\myrho}\right) - \left(\load - \balls \loadnorm\right)  \\
    &= \load \left(\ln\left(\frac{1}{\myrho}\right) - 1 + \myrho\right)  \\
    &= \load \left(\ln\left(\frac{1}{e \cdot \myrho}\right) + \myrho\right)  .
  \end{align*}
  Therefore,
  \begin{align}
    \Pr[\bino(\balls, \loadnorm) < \load]
    &\leq \exp\left(- \load \left( \myrho - 1 -  \ln(\myrho) \right) \right) .  \label{eqn:chernoff-rhs}
  \end{align}
  We now argue that the bound (\ref{eqn:chernoff-rhs}) holds not just for integer $\load$ but also any real $\load \geq 1$.
  Let $\delta_\load$ denote the right side of (\ref{eqn:chernoff-rhs}).
  Then $\delta_{\load}$ is decreasing in $\load$, because $\myrho$ is increasing in $\load$.
  So we have $\Pr[\bino(\balls,\loadnorm) < \load] = \Pr[\bino(\balls,\loadnorm) < \lceil \load \rceil] \leq \delta_{\lceil \load \rceil} \leq \delta_{\load}$, as desired.
\end{proof}

\begin{proof}[Proof of Corollary \ref{cor:rho-phase2}]
  The upper bound follows directly by plugging $\myrho \leq \left(\frac{\delta^{1/\load}}{e}\right)$ into Corollary \ref{cor:rho-tail}.
  For the lower bound, for $\myrho \geq e^2$, we have $\myrho - \ln(e \cdot \myrho) \geq \myrho/2$.
  For $\myrho \geq 2\ln\left(\tfrac{1}{\delta^{1/k}}\right)$, we have
  \begin{align*}
    \exp\left(-k \left(\frac{\myrho}{2}\right) \right)
    &\leq \delta .
  \end{align*}
\end{proof}

\begin{proof}[Proof of Corollary \ref{cor:tail-bounds}]
  Upper bound: let $\load = \left(\tfrac{e}{\delta}\right) \load^*$.
  Then, using that $\load \geq \load^*$,
  \begin{align*}
    \myrho &=    \frac{\balls \loadnorm}{\load} \\
           &\leq \frac{\balls \|\dist\|_{\load^*}}{\load}  \\
           &=    \frac{\delta \cdot \balls \|\dist\|_{\load^*}}{e \cdot \load^*}  \\
           &=    \frac{\delta}{e}  \\
           &\leq \frac{\delta^{1/\load}}{e} ,
  \end{align*}
  and we apply Corollary \ref{cor:rho-phase2}.

  Lower bound: let $\load = \load^* / \max\left\{e^2, 2\ln(\tfrac{1}{\delta})\right\}$.
  Then
  \begin{align*}
    \myrho &=    \frac{\balls \loadnorm}{\load}  \\
           &\geq \frac{\balls \|\dist\|_{\load^*}}{\load}  \\
           &=    \max\left\{e^2, ~2\ln\left(\frac{1}{\delta}\right) \right\}  \\
           &\geq \max\left\{e^2, ~2\ln\left(\frac{1}{\delta^{1/\load}}\right) \right\} ,
  \end{align*}
  and we again apply Corollary \ref{cor:rho-phase2}.
\end{proof}

\begin{proof}[Proof of Corollary \ref{cor:tail-bounds-waiting}]
  Follows almost immediately from Corollary \ref{cor:rho-phase2} and $\delta \leq \delta^{1/\load}$.
\end{proof}

\subsection{The Expected Waiting Time}
We first show that Theorem \ref{thm:main} implies Corollary \ref{cor:expected-wait}, i.e. $\frac{\load}{\load+1} \frac{\load}{e \loadnorm} \leq \E \wait \leq \frac{\load}{\loadnorm}$.
We then discuss the tightness of the bounds.

\begin{proof}[Proof of Corollary \ref{cor:expected-wait}]
  Let the random variable $X_{\load}$ be the number of independent flips of a coin with bias $\loadnorm$ until $\load$ heads appear.
  Then $\Pr[X_{\load} \leq \balls] = \Pr[\bino(\balls,\loadnorm) \geq \load]$.
  By Theorem \ref{thm:main}, $\Pr[X_{\load} \leq \balls] \leq \Pr[\wait \leq \balls]$.
  In other words, $X_{\load}$ first-order stochastically dominates $\wait$, and it follows that $\E \wait \leq \E X_{\load} = \frac{\load}{\loadnorm}$.

  For the other direction: let $\alpha = \left(\tfrac{e \cdot \loadnorm}{\load}\right)$ and let $M = \frac{1}{\alpha}$.
  Recall from Theorem \ref{thm:main} and the inequality ${\balls \choose \load} \leq \left(\tfrac{e \cdot \balls}{\load}\right)^{\load}$ that we have $\Pr[\wait \leq \balls] \leq \left(\tfrac{e \cdot \balls \cdot \loadnorm}{\load}\right)^{\load} = \balls^{\load} \alpha^{\load}$.
  \begin{align*}
    \E \wait
    &=    \sum_{\balls=1}^{\infty} \balls \Pr[\wait = \balls]  \\
    &=    \sum_{\balls=1}^{\infty} \Pr[\wait \geq \balls]  \\
    &=    \sum_{\balls=0}^{\infty} \Pr[\wait > \balls]  \\
    &=    \sum_{\balls=0}^{\infty} 1 - \Pr[\wait \leq \balls]  \\
    &\geq \sum_{\balls=0}^{\infty} \max\left\{0, 1 - \balls^{\load} \alpha^{\load}\right\}  \\
    &\geq \int_{\balls=0}^{M} \left(1 - \balls^{\load} \alpha^{\load}\right) d\balls  \\
    &=    M - \frac{M^{\load+1} \alpha^{\load}}{\load+1}  \\
    &=    M - \frac{M}{\load+1}  \\
    &=    \frac{\load}{\load+1}M  \\
    &=    \left(\frac{\load}{e \cdot (\load+1)}\right) \frac{\load}{\loadnorm} .
  \end{align*}
\end{proof}

We present Corollary \ref{cor:expected-wait} as an exploration of the tightness of Theorem \ref{thm:main},
and show now that the denumerators of  the bounds for $\E T_\load$ are (essentially) best possible.
 For $\E\wait$ a simple integral represention is known (\cite{flajolet1992birthday}):
\[ \E T_\load=\int_0^\infty \prod_{i=1}^n \big(q_\load(p_it)e^{-p_it}\big)dt\;\;,\]
where $q_\load(t):=\sum_{i=0}^{k-1}\frac{t^i}{i!}$ is the sum of the first $\load $ terms of the exponential series.
Using this we show:
\begin{prop} Let $\load\geq 2$ be fixed. Then\\  
(1) $\loadnorm \E\wait \geq {c_\load}$, and $c_\load\coloneqq\int_0^\infty e^{-\frac{t^{\load}}{\load!}}\,dt=(\load!)^{1/\load} \Gamma(1+\frac{1}{\load}) $ is the best possible lower bound.\\
(2) $\loadnorm \E\wait \leq {\load}$, and $\load$ is the best possible upper bound.
\end{prop}
\begin{proof} It is easy to see (using logarithmic differentiation) that  $q_\load(t)e^{-t}\geq e^{-\frac{t^\load}{\load!}}$ for $t\geq 0$.\\ 
(1) Therefore
\[ \loadnorm\E T_\load \geq \loadnorm\int_0^\infty \prod_{i=1}^n \big(e^{-\frac{p_i^\load t^\load}{\load!}}\big)dt= \loadnorm\int_0^\infty e^{-\frac{\loadnorm^\load t^\load}{\load!}} dt= c_\load\;\;\]
Thus $c_\load$ is a lower bound. It is best possible since it arises as the limit of $\loadnorm\E T_\load $
for the uniform distribution on bins as $\bins\longrightarrow \infty$  (\cite{klamkin1967extensions}) .\\
(2) We now already know that $\load$ is an upper bound.  Let $a\in(0,1)$ and let $p_1=a$, $p_2=\ldots=p_\bins=\frac{1-a}{n-1}$.  Then $\loadnorm^k=a^k+\frac{(1-a)^\load}{(\bins-1)^{\load-1}}$ and 
\[ \E T_\load =\int_0^\infty \prod_{i=1}^n \big(q_\load(p_it)e^{-p_it}\big)dt\geq \int_0^\infty  q_\load(at)e^{-at} e^{\frac{(1-a)^\load t^\load}{(\bins-1)^{\load-1} \load!}}dt\]
For $\bins \longrightarrow \infty$ then $\loadnorm \longrightarrow a$ and $\E T_\load\longrightarrow \frac{k}{a}$. Therefore the bound $k$ can not be lowered. 
\end{proof}
Thus the upper bound of Corollary \ref{cor:expected-wait} is best possible, and since 
$c_\load>\frac{\load}{e}$ and 
  \[ c_\load=\frac{\load}{e}\Big(1+\frac{1}{2}\frac{\log2\pi \load}{\load}-\frac{\gamma}{\load}+O\left(\frac{1}{\load^2}\right)\Big) , \]
also the lower bound of  Corollary \ref{cor:expected-wait} is essentially the best possible.

\section{Generalizations} \label{sec:inf}

Finally we remark that Theorem \ref{thm:main} can be generalized in two directions.
\paragraph{(1) Maximum load in a restricted set of bins.} Suppose that we are only interested in the maximum load (resp. the waiting time for maximum load $\load$) of $ 
\binsr<\bins$ bins (say, bins $1,\ldots,\binsr$) , where $p_1+\ldots+p_\binsr<1$.  Denote by $\maxloadr$ the maximum load in bins $1,\ldots,\binsr$ at time $\balls$, and let $\loadnormr$ denote the $\load$-norm of $(p_1,\ldots,p_\binsr)$. Then
    \[ \Pr[\bino(\balls, \loadnormr) \geq \load] ~\leq~ \Pr[\maxloadr \geq \load] ~\leq~ {\balls \choose \load}\loadnormr^{\load} . \]
\begin{proof} The proof for the upper bound applies essentially unchanged. For the lower bound we may assume that there is only one additional bin (numbered $\binsr +1$), which is chosen with probability $p_{\binsr+1}=1-(p_1+\ldots+p_\binsr)$. Let $\loadrp$ denote the random variable ``load of bin $\binsr+1$ after $\balls$ balls". Then $\loadrp$ is binomially distributed with parameters $\balls$ and $p_{\binsr+1}$, and conditionally on $\loadrp=\ell$ the joint distribution of the loads in bins $1,\ldots,\binsr$ is multinomial with parameters $\balls -\ell$ and $q_1=\frac{p_1}{1-p_{\binsr+1}},\ldots, q_r=\frac{p_r}{1-p_{\binsr+1}}$. Conditioning on $\loadrp$, and using Theorem \ref{thm:main}, we get
\begin{align*}
\Pr[ \maxloadr \geq \load]&=\sum_{\ell=0}^\balls \Pr[ \maxloadr \geq \load\mid \loadrp=\ell] \Pr[\loadrp=\ell]\\
&\geq  \sum_{\ell=0}^{\balls}  \Pr[\bino(\balls-\ell, \qnorm) \geq \load] \Pr[\loadrp=\ell]\\
&=\sum_{\ell=0}^\balls \sum_{ j=\load  }^{\balls -\ell} {\balls \choose \ell} p_{\binsr+1}^\ell(1-p_{\binsr+1})^{\balls - \ell}{\balls-\ell \choose j} \qnorm^j(1-\qnorm)^{\balls -\ell -j}\\
&=\sum_{j=\load}^\balls \sum_{\ell=0}^{\balls-j} \frac{\balls!}{j!\ell!( \balls-\ell-j)!} p_{\binsr+1}^\ell \loadnormr^j (1-p_{\binsr+1}-\loadnormr)^{\balls -\ell -j} \\
&=\sum_{j=\load}^\balls {\balls \choose j} \loadnormr^j(1-\loadnormr)^{\balls-j}\\
&=  \Pr[\bino(\balls, \loadnormr) \geq \load]
\end{align*}
\end{proof} 

\paragraph{(2)  Maximum load in countably infinite many bins.} Let $p=p_1,p_2,...$ a probability distribution with infinite support.
Then Theorem \ref{thm:main} still holds, i.e. 
    \[ \Pr[\bino(\balls, \loadnorm) \geq \load] ~\leq~ \Pr[\maxload \geq \load] ~\leq~ {\balls \choose \load}\loadnorm^{\load} . \]
    \begin{proof}  By the preceding remark we have that for any $\binsr\geq 1$
   \[ \Pr[\bino(\balls, \loadnormr) \geq \load] ~\leq~ \Pr[\maxloadr \geq \load] ~\leq~ {\balls \choose \load}\loadnormr^{\load}  \] 
   Passing to the limit 
   $r\longrightarrow \infty$ proves the claim.
\end{proof}

Applying the argument (1) again to the countably-infinite setting gives the generalization in both directions simultaneously.

\bibliographystyle{plainnat}
\bibliography{citations}

\end{document}